\documentclass{amsart}

\usepackage{graphicx,amsthm, amsmath, amssymb, mathrsfs}

\usepackage{graphicx}
\usepackage{bbm}
\usepackage{tikz}

\newtheorem{theorem}{Theorem}[section]

\newtheorem{lemma}{Lemma}[section]

\theoremstyle{definition}

\newtheorem{cor}[theorem]{Corollary}

\numberwithin{equation}{section}

\def\XXint#1#2#3{{\setbox0=\hbox{$#1{#2#3}{\int}$}
     \vcenter{\hbox{$#2#3$}}\kern-.5\wd0}}

\makeatletter
\def\moverlay{\mathpalette\mov@rlay}
\def\mov@rlay#1#2{\leavevmode\vtop{%
   \baselineskip\z@skip \lineskiplimit-\maxdimen
   \ialign{\hfil$\m@th#1##$\hfil\cr#2\crcr}}}
\newcommand{\charfusion}[3][\mathord]{
    #1{\ifx#1\mathop\vphantom{#2}\fi
        \mathpalette\mov@rlay{#2\cr#3}
      }
    \ifx#1\mathop\expandafter\displaylimits\fi}
\makeatother

\newcommand{\Sp}{\mathcal{S}}
\newcommand{\D}{\mathscr{D}}
\newcommand{\R}{\mathbb{R}}

\DeclareMathOperator{\supp}{supp}

\begin{document}

\title[A counter example for a sparse operator]{Muckenhoupt-Wheeden conjectures for sparse operators}

\author{Cong Hoang}
\address{Cong Hoang, Department of Mathematics, University of Alabama, Tuscaloosa, AL 35487-0350}

%    General info

\author{Kabe Moen}
\address{Kabe Moen, Department of Mathematics, University of Alabama, Tuscaloosa, AL 35487-0350}
%\curraddr{Department of Mathematics and Statistics,
%\email{xyz@math.university.edu}
%\thanks{The first author was supported by NSF Grant \#1201504.}

%\date{April 21, 2016}



\maketitle

\begin{abstract}
We provide an explicit example of a pair of weights and a dyadic sparse operator for which the Hardy-Littlewood maximal function is bounded from $L^p(v)$ to $L^p(u)$ and from $L^{p'}(u^{1-p'})$ to $L^{p'}(v^{1-p'})$ while the  sparse operator is not bounded on the same spaces.  Our construction also provides an example of a single weight for which the weak-type endpoint does not hold for sparse operators.

\hfill

\end{abstract}

%-------------------------------------------------
% NEW SECTION: INTRODUCTION & STATEMENT OF MAIN RESULTS
%-------------------------------------------------
\section{Introduction and Statement of Main Results}

We are interested in the Muckenhoupt and Wheeden conjectures that relate weighted inequalities for the Hardy-Littlewood maximal operator and those for Calder\'on-Zygmund operators.   In one dimension, Reguera-Scurry \cite{RS} showed that there exists a pair of weights $(u,v)$ for which
\begin{equation}\label{maxbd}
M:L^p(v)\rightarrow L^p(u) \quad \text{and} \quad M:L^{p'}(u^{1-p'})\rightarrow L^{p'}(v^{1-p'})
\end{equation}
but at the same time the Hilbert transform is not bounded from $L^p(v)$ to $L^p(u)$.  Criado and F. Soria \cite{CS} extended this to higher dimensions by constructing a pair of weights for which \eqref{maxbd} held but that any Calder\'on-Zygmund operator in $\R^n$ associated to a certain non-degenerate kernel was not bounded from $L^p(v)$ to $L^p(u)$.  Specifically, it is shown in \cite{CS} that if $T$ is a non-degenerate Calder\'on-Zygmund operator, then there exists a pair of weights $(u,v)$ for which \eqref{maxbd} holds and a function $f\in L^p(v)$ such that $\|Tf\|_{L^p(u)}=\infty$.  Moreover, it turns out that the same construction of weights can be used to produce a weight $w$ for which $T$ is unbounded from $L^1(Mw)$ to $L^{1,\infty}(w)$.  This was shown first by Reguera \cite{R} for Haar shift operators and then by Reguera and Thiele for the Hilbert transform \cite{RT} and finally for non-degenerate Calder\'on-Zygmund operators by Criado and Soria \cite{CS}.

On a parallel note, sparse operators have turned out to be very important tools in harmonic analysis.  Let $\D$ be a dyadic grid in $\R^n$.  A collection $\mathcal{S}$ of cubes from $\D$ is called a sparse family if for any cube $Q\in\mathcal{S}$, there exist a subset $E_Q\subseteq Q$ such that $|Q|\leqslant2|E_Q|$, and the collection $\{E_Q\}_{Q\in\mathcal{S}}$ is pairwise disjoint.  Lerner and Nazarov \cite{LN} showed that this condition is equivalent to the Carleson condition:
$$\sum_{\substack{P\in \Sp \\ P\subset Q}} |P|\leq C|Q|, \qquad Q\in \D.$$

Given a sparse family $\mathcal{S}$, a sparse operator $T_\mathcal{S}$ is defined by
$$
T_\mathcal{S}f(x)=\sum_{Q\in \mathcal{S}}\frac{1}{|Q|}\int_Qf(y)\thinspace dy \thinspace \mathbbm{1}_Q(x).
$$
Lerner \cite{L} showed that if $T$ is a Calder\'on-Zygmund operator then  
$$
\|T\|_{L^p(v)\rightarrow L^p(u)}\lesssim \sup_\mathcal{S}\|T_\mathcal{S}\|_{L^p(v)\rightarrow L^p(u)}
$$
where the supremum is over all sparse families from a finite number of dyadic grids. The results of Reguera-Scurry and Criado-Soria and the above inequality implies that there exists a dyadic grid $\D$ such that
\begin{equation}\label{supinfty}\sup_\mathcal{S\subseteq \D}\|T_\mathcal{S}\|_{L^p(v)\rightarrow L^p(u)}=\infty\end{equation} 
where the supremum is over all sparse families from that dyadic grid.  However, this does not imply that there exist a single sparse family $\mathcal{S}$ for which the operator $T_\mathcal{S}$ is unbounded from $L^p(v)\rightarrow L^p(u)$.

Recently, Conde-Alonso and Rey \cite{CAR}; Lerner and Nazarov \cite{LN}; Lacey \cite{La}; Hyt\"onen, Roncal, and Tapiola \cite{HRT}, and Lerner \cite{L2} proved that Calder\'on-Zygmund operators are bounded pointwise by finitely many sparse operators.  Specifically, given a nice function $f$ there exist finitely many dyadic grids $\D^1,\ldots,\D^N$ and sparse families $\Sp^i\in \D^i$ such that
\begin{equation}\label{lacey}|Tf|\lesssim \sum_{i=1}^N T_{S^i}|f|.\end{equation}
Combining inequality \eqref{lacey} with the example of Criado and Soria shows that there is a pair of weights $(u,v)$ and a sparse operator $T_\Sp$ such that \eqref{maxbd} holds but there is a function $f\in L^p(v)$ for which $\|T_\Sp f\|_{L^p(u)}=\infty$.  Moreover, there exist a weight $w$ and a sparse operator $T_\Sp$ which is unbounded from $L^1(Mw)$ to $L^{1,\infty}(w)$.  However, the explicit sparse family $\Sp$ cannot be directly computed.  The sparse families from \eqref{lacey} depend on $Tf$ and are not explicit, rather they come from a stopping time argument based on the level sets of $f $.  We would also like to point readers to the recent work of Culiuc, Di Plinio, and Ou \cite{CDPO} who show that sparse domination of a bilinear form can be achieved with an explicit construction that is based upon the level sets of $Mf$ rather than $Tf$.

The main purpose of this note is to find an exact sparse family in $\mathbb{R}^n$ that answers the question. In fact we provide a simple sparse family for which the associated sparse operator is unbounded.  In addition, our example is purely dyadic and we do not need to exploit any cancelation properties of the operator.  The sparse family that we consider here is $\mathcal{S}=\bigcup_{k=0}^{\infty} \mathcal{S}_k$ in which
\begin{equation*}
\begin{split}
    & \mathcal{S}_0=\{[0,2^{-N})^n:\hspace{1mm} N=0,1,2,...\}, \\
    & \mathcal{S}_k=\{[2^k,2^k+2^{-N})^n:\hspace{1mm} N=0,1,2,...\} \text{ for } k\geqslant1.
\end{split}
\end{equation*}

Our main result is the following theorem..

\begin{theorem} \label{thmA}
For every $p\in (1,\infty)$, there exists a pair of weights $(u,v)$ such that
$$
M:\hspace{2mm}L^p(v)\rightarrow L^p(u)
$$
$$
M:\hspace{2mm}L^{p'}(u^{1-p'})\rightarrow L^{p'}(v^{1-p'})
$$
while $T_\mathcal{S}$ is unbounded from $L^p(v)$ to $L^p(u)$.
\end{theorem}

As we will see later that $v\sim u$ a.e. on the support of $u$, the proof of Theorem \ref{thmA} can be modified to provide a counterexample in one-weight settings with only minor modifications needed. We note here that the weights mentioned in this paper vanish on the sets of positive measure. We have the following result.

\begin{theorem} \label{thmB}
For every $p\in (1,\infty)$, there exists a weight $w$ such that
$$
M:\hspace{2mm}L^p(w)\rightarrow L^p(w)
$$
while $T_\mathcal{S}$ is unbounded from $L^p(w)$ to $L^p(w)$.
\end{theorem}

As an immediate consequence of Theorem \ref{thmA} and an extrapolation argument of C. Per\'ez and D. Cruz-Uribe \cite{PC}, we have the following result.

\begin{cor} \label{thmC}
There exists a weight $w$ such that
$T_\mathcal{S}$ is unbounded from $L^1(Mw)$ to $L^{1,\infty}(w)$.
\end{cor}

%-------------------------------------------------
% NEW SECTION: PRELIMINARIES
%-------------------------------------------------
\section{Preliminaries}

Since we will be working with weights that vanish on sets of positive measure, we will define $L^p(w)$ to be all functions whose support is contained in the support of $w$ and $\|f\|_{L^p(w)}<\infty$.  This convention is also used in \cite{CS}.  Unlike the Hilbert transform and other Calder\'on-Zygmund operators, sparse operators do not have any cancellation properties to exploit. For this reason, we are going to construct a purely dyadic weight.  %that is in fact an $n$-dimensional dyadic version of the ones appearing in \cite{CS}, \cite{RS}, and \cite{RT}. 
This construction is to reduce unnecessary complications of the weight and, at the same time, simplify our later calculations.

\hfill

For each integer $k$, we start with the unit cube $P^0_1=[0,1)^n$ and a weight $w_k^0$ supported on it with $w_k^0(P^0_1)=1$ (see Figure \ref{fig1}). We will gradually modify this weight to obtain the desired one.

\hfill

\begin{figure}
\setlength{\unitlength}{1.85mm}
\begin{picture}(64,7)
\put(0,0){\line(1,0){64}}
\put(0,5){\line(1,0){64}}
\put(0,-0.5){\line(0,1){1}}
\put(64,-0.5){\line(0,1){1}}
\put(-1,-2){0}
\put(64.2,-2){1}
%\put(7,-5){$\textit{Figure 1 -- A visualization of } \hspace{1mm} w_k^0 \textit{ when } n=1, \textit{ and } \thinspace k=3$.}
\end{picture}
\vspace{3mm}

\caption{A visualization of $w_k^0$ when $n=1$ and $k=3$}
\label{fig1}
\end{figure}
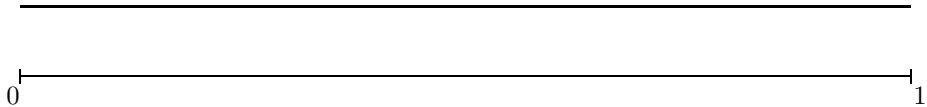

\textbf{Step 1:} We equally decompose the unit cube into $2^{kn}$ sub-cubes and consider the collection $\mathscr{P}^1=\bigl\{2^{-k}\bigl((z_1,...,z_n)+[0,1)^n\bigr): \hspace{1mm} z_i=0,1,...,2^{k-1}-1\bigr\}$. The union of this collection is, in fact, $[0,\frac{1}{2})^n$. We can reindex the cubes in this collection to be $\mathscr{P}^1=\{P^1_i: \hspace{1mm} i=1,2,...,2^{(k-1)n}\}$, in which $P^1_1=[0,2^{-k})^n$. Let $Q_1^1=[\frac{1}{2},\frac{1}{2}+2^{-k})^n$, and $\mathscr{Q}^1=\{Q^1_1\}$. We then define $w_k^1$ to be uniformly distributed on $[0,\frac{1}{2})^n\cup Q_1^1$ so that $w_k^1\bigl([0,\frac{1}{2})^n\cup Q_1^1\bigr)=w_k^0([0,1)^n)=1$ (see Figure \ref{fig2}).

\hfill

\begin{figure}
\setlength{\unitlength}{1.85mm}
\begin{picture}(64,11)
\put(0,0){\line(1,0){64}}
\put(0,8){\line(1,0){40}}
\put(64,-0.5){\line(0,1){1}}
\multiput(0,-0.5)(8,0){6}{\line(0,1){1}}
\linethickness{0.9mm}
\put(32,0){\line(1,0){8}}
\put(-1,-2){0}
\put(64.2,-2){1}
\put(31.4,-3){$\frac{1}{2}$}
\put(39.4,-3){$\frac{1}{2}+2^{-k}$}
\put(36,2){$Q^1_1$}
%\put(7,-7){$\textit{Figure 2 -- A visualization of } \hspace{1mm} w_k^1 \textit{ when } n=1, \textit{ and } \thinspace k=3$.}
\end{picture}
\vspace{3mm}
\caption{A visualization of $w_k^1$ when $n=1$ and $k=3$}
\label{fig2}
\end{figure}
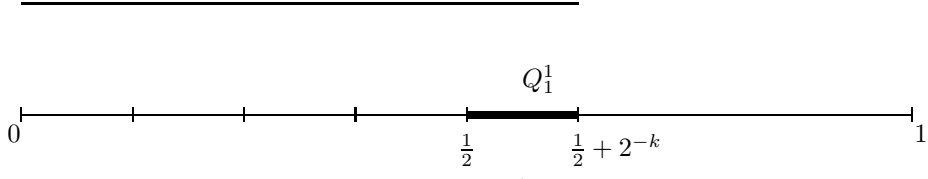

\textbf{Step 2:} For each cube in the collection $\mathscr{P}^1$, say $P^1_i$, we will treat it just as a smaller scale of the unit cube. To be more precise, we let Ver($P^1_i$) denote the vertex of $P^1_i$ that is closest to the origin, and consider the collection Ver$(P^1_i)+2^{-k}\bigl(\mathscr{P}^1\cup\{Q^1_1\}\bigr)$ of smaller cubes contained inside $P^1_i$. We then iterate the process of step 1 into $P^1_i$. Let $\mathscr{P}^2=\bigcup_{i=1}^{2^{(k-1)n}} \left(\text{Ver}(P^1_i)+2^{-k}\mathscr{P}^1\right)$ and $\mathscr{Q}^2=\left\{\text{Ver}(P^1_i)+2^{-k}Q^1_1:\hspace{1mm} i=1,2,...,2^{(k-1)n}\right\}$. We reindex these collections so that $$\mathscr{P}^2=\{P^2_i:\hspace{1mm} i=1,2,...,2^{(k-1)2n}\}\quad
\text{and} \quad \mathscr{Q}^2=\{Q^2_i:\hspace{1mm} i=1,2,...,2^{(k-1)n}\}$$ 
where $Q^2_1=2^{-k}\thinspace Q^1_1=[\frac{1}{2}2^{-k},\frac{1}{2}2^{-k}+2^{-2k})^n$. We then define $w^2_k$ to be uniformly distributed, so that $w_k^2\left(\bigcup_{R\in\mathscr{P}^2\cup\mathscr{Q}^2}R\right)=w_k^1\left(\bigcup_{R\in\mathscr{P}^1}R\right)$ (see Figure \ref{fig3}).

\hfill

\setlength{\unitlength}{1.85mm}
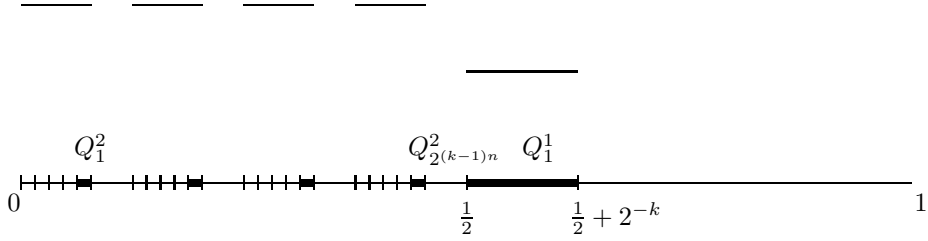
\begin{figure}
\begin{picture}(64,16)
\put(0,0){\line(1,0){64}}
\put(32,8){\line(1,0){8}}
\multiput(0,12.8)(8,0){4}{\line(1,0){5}}
\multiput(0,-0.5)(8,0){6}{\line(0,1){1}}
\multiput(0,0)(8,0){4}{\multiput(1,-0.5)(1,0){5}{\line(0,1){1}}}
\linethickness{0.9mm}
\put(32,0){\line(1,0){8}}
\multiput(4,0)(8,0){4}{\line(1,0){1}}
\put(-1,-2){0}
\put(64.2,-2){1}
\put(31.4,-3){$\frac{1}{2}$}
\put(39.4,-3){$\frac{1}{2}+2^{-k}$}
\put(36,2){$Q^1_1$}
\put(3.8,2){$Q^2_1$}
\put(27.8,2){$Q^2_{2^{(k-1)n}}$}
%\put(7,-7){$\textit{Figure 3 -- A visualization of } \hspace{1mm} w_k^2 \textit{ when } n=1, \textit{ and } \thinspace k=3$.}
\end{picture}
\vspace{3mm}
\caption{A visualization of $w_k^2$ when $n=1$ and $k=3$}
\label{fig3}
\end{figure}

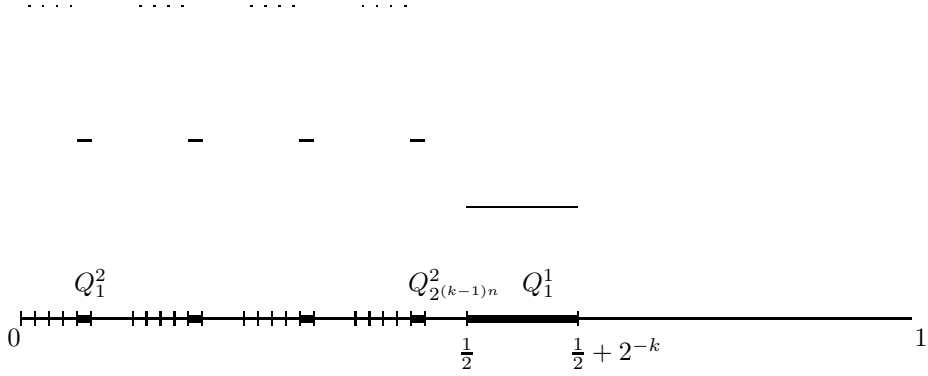
\begin{figure}
\setlength{\unitlength}{1.85mm}
\begin{picture}
(64,25)
\put(0,0){\line(1,0){64}}
\put(32,8){\line(1,0){8}}
\multiput(4,12.8)(8,0){4}{\line(1,0){1}}
\multiput(0,0)(8,0){4}{\multiput(0.5,22.48)(1,0){4}{\line(1,0){0.125}}}
\multiput(0,-0.5)(8,0){6}{\line(0,1){1}}
\multiput(0,0)(8,0){4}{\multiput(1,-0.5)(1,0){5}{\line(0,1){1}}}
\linethickness{0.9mm}
\put(32,0){\line(1,0){8}}
\multiput(4,0)(8,0){4}{\line(1,0){1}}
\put(-1,-2){0}
\put(64.2,-2){1}
\put(31.4,-3){$\frac{1}{2}$}
\put(39.4,-3){$\frac{1}{2}+2^{-k}$}
\put(36,2){$Q^1_1$}
\put(3.8,2){$Q^2_1$}
\put(27.8,2){$Q^2_{2^{(k-1)n}}$}
%\put(7,-7){$\textit{Figure 4 -- A visualization of } \hspace{1mm} w_k \textit{ when } n=1, \textit{ and } \thinspace k=3$.}
\end{picture}
\vspace{3mm}
\caption{A visualization of $w_k$ when $n=1$ and $k=3$}
\label{fig4}
\end{figure}

We may recursively build a sequence of weights $\{w_k\}$ so that each $w_k([0,1]^n)=1$.   The family of weights $w_k$ (see Figures ~\ref{fig4} and \ref{fig5}) is  supported on $\bigcup_{m=1}^{\infty}\bigcup_{i=1}^{2^{(k-1)(m-1)n}}Q_i^m$, and by induction we have

$$w_k(x)=\sum_{m=1}^{\infty}\left(\frac{2^{kn}}{2^{(k-1)n}+1}\right)^m\mathbbm{1}_{\Omega_m}(x)$$
where $\Omega_m=\bigcup_{i=1}^{2^{(k-1)(m-1)n}}Q^m_i$, $|Q_i^m|=2^{-kmn}$, and $w_k(Q_i^m)=(2^{(k-1)n}+1)^{-m}$.  These sets depend on $k$, but we suppress the index $k$ for simplicity.  We will write $Q^m_i(k)$ and $\Omega_m(k)$ when the dependence on $k$ is needed.  Finally, we also note that the family of cubes $\{Q^m_i\}$ is disjoint.  This fact will be used later in our calculations.

\hfill

Completely similar to what has been proved in \cite{CS} and \cite{RS}, we have the following lemma.

\begin{lemma} \label{lem1}
Let $M$ be the Hardy-Littlewood maximal function, then for all $x\in \cup_{m=1}^{\infty}\Omega_m$, we have
$$M(w_k)(x)\leqslant 9^n \thinspace w_k(x).$$
\end{lemma}

In this paper, we will make use of the following well known result: for any pair of weights $(v,w)$, the followings are equivalent.
\begin{equation*}
\begin{split}
    \|Tf\|_{L^p(w)} & \lesssim \|f\|_{L^p(v)} \\
    \|T^*f\|_{L^{p'}(v^{1-p'})} & \lesssim \|f\|_{L^{p'}(w^{1-p'})}
\end{split}    
\end{equation*}
where we define $v^{1-p'}=v^{1-p'}\mathbbm{1}_{\supp(v)}$ and $w^{1-p'}=w^{1-p'}\mathbbm{1}_{\supp(w)}$.

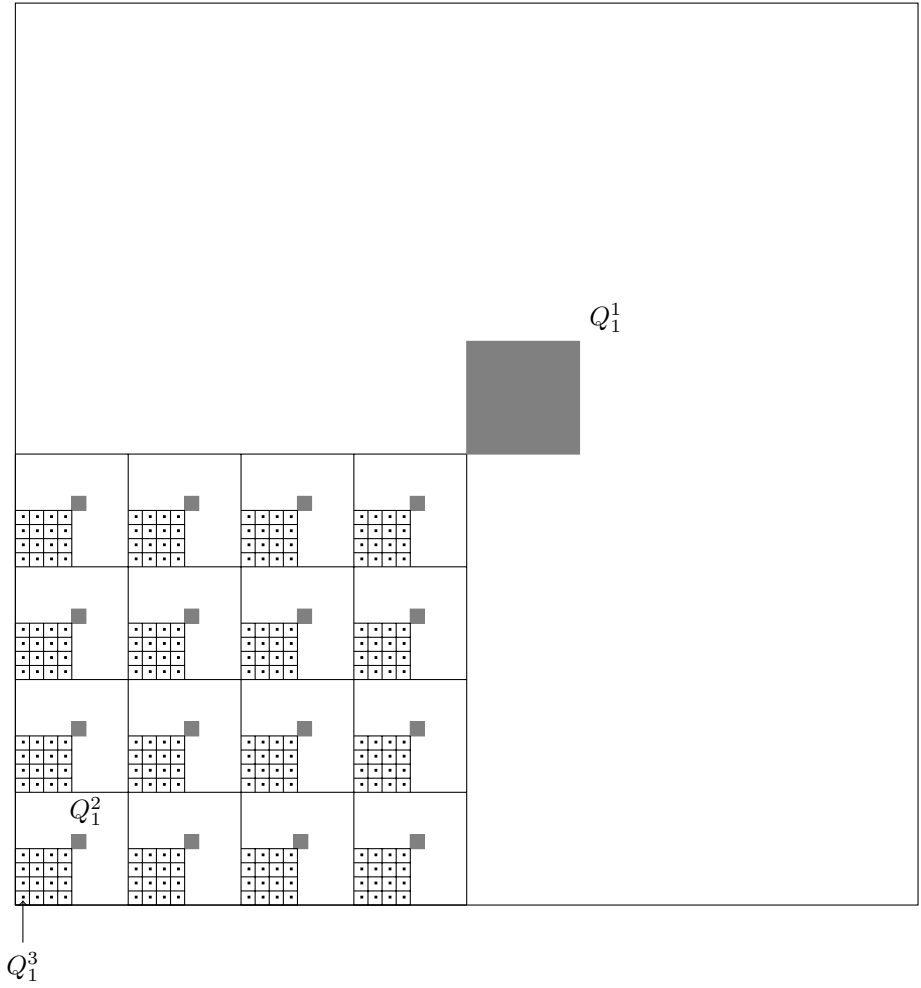
\begin{figure}
\begin{tikzpicture}
%\draw[black, thin] %(3.5mm,-12mm) node[anchor= west] {$\textit{Figure 5 -- A visualization of the support of} \hspace{2mm} w_k \textit{ when } n=2, \textit{ and } \thinspace k=3$.};
\draw[black, thin] (0mm,0mm) rectangle (120mm,120mm);
\draw[step=15mm, black, thin](0,0) grid (60mm,60mm);
\filldraw[gray, thin] (60mm,60mm) rectangle +(15mm,15mm) node[black, thick, anchor=south west] {$Q_1^1$};
\filldraw[gray, thin] (7.5mm,7.5mm) rectangle (9.375mm,9.375mm) node[black, thick, anchor=south] {$Q_1^2$};
\filldraw[gray, thin] (22.5mm,7.5mm) rectangle +(1.875mm,1.875mm);
\filldraw[gray, thin] (37.mm,7.5mm) rectangle +(1.875mm,1.875mm);
\filldraw[gray, thin] (52.5mm,7.5mm) rectangle +(1.875mm,1.875mm);
\filldraw[gray, thin] (7.5mm,22.5mm) rectangle +(1.875mm,1.875mm);
\filldraw[gray, thin] (7.5mm,37.5mm) rectangle +(1.875mm,1.875mm);
\filldraw[gray, thin] (7.5mm,52.5mm) rectangle +(1.875mm,1.875mm);
\filldraw[gray, thin] (22.5mm,22.5mm) rectangle +(1.875mm,1.875mm);
\filldraw[gray, thin] (22.5mm,37.5mm) rectangle +(1.875mm,1.875mm);
\filldraw[gray, thin] (22.5mm,52.5mm) rectangle +(1.875mm,1.875mm);
\filldraw[gray, thin] (37.5mm,22.5mm) rectangle +(1.875mm,1.875mm);
\filldraw[gray, thin] (37.5mm,37.5mm) rectangle +(1.875mm,1.875mm);
\filldraw[gray, thin] (37.5mm,52.5mm) rectangle +(1.875mm,1.875mm);
\filldraw[gray, thin] (52.5mm,22.5mm) rectangle +(1.875mm,1.875mm);
\filldraw[gray, thin] (52.5mm,37.5mm) rectangle +(1.875mm,1.875mm);
\filldraw[gray, thin] (52.5mm,52.5mm) rectangle +(1.875mm,1.875mm);
\draw[step=1.875mm, black, thin](0mm,0mm) grid +(7.5mm,7.5mm);
\draw[step=1.875mm, black, thin](0mm,15mm) grid +(7.5mm,7.5mm);
\draw[step=1.875mm, black, thin](0mm,30mm) grid +(7.5mm,7.5mm);
\draw[step=1.875mm, black, thin](0mm,45mm) grid +(7.5mm,7.5mm);
\draw[step=1.875mm, black, thin](15mm,0mm) grid +(7.5mm,7.5mm);
\draw[step=1.875mm, black, thin](15mm,15mm) grid +(7.5mm,7.5mm);
\draw[step=1.875mm, black, thin](15mm,30mm) grid +(7.5mm,7.5mm);
\draw[step=1.875mm, black, thin](15mm,45mm) grid +(7.5mm,7.5mm);
\draw[step=1.875mm, black, thin](30mm,0mm) grid +(7.5mm,7.5mm);
\draw[step=1.875mm, black, thin](30mm,15mm) grid +(7.5mm,7.5mm);
\draw[step=1.875mm, black, thin](30mm,30mm) grid +(7.5mm,7.5mm);
\draw[step=1.875mm, black, thin](30mm,45mm) grid +(7.5mm,7.5mm);
\draw[step=1.875mm, black, thin](45mm,0mm) grid +(7.5mm,7.5mm);
\draw[step=1.875mm, black, thin](45mm,15mm) grid +(7.5mm,7.5mm);
\draw[step=1.875mm, black, thin](45mm,30mm) grid +(7.5mm,7.5mm);
\draw[step=1.875mm, black, thin](45mm,45mm) grid +(7.5mm,7.5mm);

\filldraw[black, thin] (0.9375mm,0.9375mm) rectangle +(0.234675mm,0.234375mm);
\draw[<-] (1mm,0.6mm) -- (1mm,-5mm) node[black, thick, anchor= north] {$Q_1^3$};
\filldraw[black, thin] (0.9375mm,2.8125mm) rectangle +(0.234675mm,0.234375mm);
\filldraw[black, thin] (0.9375mm,4.6875mm) rectangle +(0.234675mm,0.234375mm);
\filldraw[black, thin] (0.9375mm,6.5625mm) rectangle +(0.234675mm,0.234375mm);
\filldraw[black, thin] (0.9375mm,15.9375mm) rectangle +(0.234675mm,0.234375mm);
\filldraw[black, thin] (0.9375mm,17.8125mm) rectangle +(0.234675mm,0.234375mm);
\filldraw[black, thin] (0.9375mm,19.6875mm) rectangle +(0.234675mm,0.234375mm);
\filldraw[black, thin] (0.9375mm,21.5625mm) rectangle +(0.234675mm,0.234375mm);
\filldraw[black, thin] (0.9375mm,30.9375mm) rectangle +(0.234675mm,0.234375mm);
\filldraw[black, thin] (0.9375mm,32.8125mm) rectangle +(0.234675mm,0.234375mm);
\filldraw[black, thin] (0.9375mm,34.6875mm) rectangle +(0.234675mm,0.234375mm);
\filldraw[black, thin] (0.9375mm,36.5625mm) rectangle +(0.234675mm,0.234375mm);
\filldraw[black, thin] (0.9375mm,45.9375mm) rectangle +(0.234675mm,0.234375mm);
\filldraw[black, thin] (0.9375mm,47.8125mm) rectangle +(0.234675mm,0.234375mm);
\filldraw[black, thin] (0.9375mm,49.6875mm) rectangle +(0.234675mm,0.234375mm);
\filldraw[black, thin] (0.9375mm,51.5625mm) rectangle +(0.234675mm,0.234375mm);

\filldraw[black, thin] (2.8125mm,0.9375mm) rectangle +(0.234675mm,0.234375mm);
\filldraw[black, thin] (2.8125mm,2.8125mm) rectangle +(0.234675mm,0.234375mm);
\filldraw[black, thin] (2.8125mm,4.6875mm) rectangle +(0.234675mm,0.234375mm);
\filldraw[black, thin] (2.8125mm,6.5625mm) rectangle +(0.234675mm,0.234375mm);
\filldraw[black, thin] (2.8125mm,15.9375mm) rectangle +(0.234675mm,0.234375mm);
\filldraw[black, thin] (2.8125mm,17.8125mm) rectangle +(0.234675mm,0.234375mm);
\filldraw[black, thin] (2.8125mm,19.6875mm) rectangle +(0.234675mm,0.234375mm);
\filldraw[black, thin] (2.8125mm,21.5625mm) rectangle +(0.234675mm,0.234375mm);
\filldraw[black, thin] (2.8125mm,30.9375mm) rectangle +(0.234675mm,0.234375mm);
\filldraw[black, thin] (2.8125mm,32.8125mm) rectangle +(0.234675mm,0.234375mm);
\filldraw[black, thin] (2.8125mm,34.6875mm) rectangle +(0.234675mm,0.234375mm);
\filldraw[black, thin] (2.8125mm,36.5625mm) rectangle +(0.234675mm,0.234375mm);
\filldraw[black, thin] (2.8125mm,45.9375mm) rectangle +(0.234675mm,0.234375mm);
\filldraw[black, thin] (2.8125mm,47.8125mm) rectangle +(0.234675mm,0.234375mm);
\filldraw[black, thin] (2.8125mm,49.6875mm) rectangle +(0.234675mm,0.234375mm);
\filldraw[black, thin] (2.8125mm,51.5625mm) rectangle +(0.234675mm,0.234375mm);

\filldraw[black, thin] (4.6875mm,0.9375mm) rectangle +(0.234675mm,0.234375mm);
\filldraw[black, thin] (4.6875mm,2.8125mm) rectangle +(0.234675mm,0.234375mm);
\filldraw[black, thin] (4.6875mm,4.6875mm) rectangle +(0.234675mm,0.234375mm);
\filldraw[black, thin] (4.6875mm,6.5625mm) rectangle +(0.234675mm,0.234375mm);
\filldraw[black, thin] (4.6875mm,15.9375mm) rectangle +(0.234675mm,0.234375mm);
\filldraw[black, thin] (4.6875mm,17.8125mm) rectangle +(0.234675mm,0.234375mm);
\filldraw[black, thin] (4.6875mm,19.6875mm) rectangle +(0.234675mm,0.234375mm);
\filldraw[black, thin] (4.6875mm,21.5625mm) rectangle +(0.234675mm,0.234375mm);
\filldraw[black, thin] (4.6875mm,30.9375mm) rectangle +(0.234675mm,0.234375mm);
\filldraw[black, thin] (4.6875mm,32.8125mm) rectangle +(0.234675mm,0.234375mm);
\filldraw[black, thin] (4.6875mm,34.6875mm) rectangle +(0.234675mm,0.234375mm);
\filldraw[black, thin] (4.6875mm,36.5625mm) rectangle +(0.234675mm,0.234375mm);
\filldraw[black, thin] (4.6875mm,45.9375mm) rectangle +(0.234675mm,0.234375mm);
\filldraw[black, thin] (4.6875mm,47.8125mm) rectangle +(0.234675mm,0.234375mm);
\filldraw[black, thin] (4.6875mm,49.6875mm) rectangle +(0.234675mm,0.234375mm);
\filldraw[black, thin] (4.6875mm,51.5625mm) rectangle +(0.234675mm,0.234375mm);

\filldraw[black, thin] (6.5625mm,0.9375mm) rectangle +(0.234675mm,0.234375mm);
\filldraw[black, thin] (6.5625mm,2.8125mm) rectangle +(0.234675mm,0.234375mm);
\filldraw[black, thin] (6.5625mm,4.6875mm) rectangle +(0.234675mm,0.234375mm);
\filldraw[black, thin] (6.5625mm,6.5625mm) rectangle +(0.234675mm,0.234375mm);
\filldraw[black, thin] (6.5625mm,15.9375mm) rectangle +(0.234675mm,0.234375mm);
\filldraw[black, thin] (6.5625mm,17.8125mm) rectangle +(0.234675mm,0.234375mm);
\filldraw[black, thin] (6.5625mm,19.6875mm) rectangle +(0.234675mm,0.234375mm);
\filldraw[black, thin] (6.5625mm,21.5625mm) rectangle +(0.234675mm,0.234375mm);
\filldraw[black, thin] (6.5625mm,30.9375mm) rectangle +(0.234675mm,0.234375mm);
\filldraw[black, thin] (6.5625mm,32.8125mm) rectangle +(0.234675mm,0.234375mm);
\filldraw[black, thin] (6.5625mm,34.6875mm) rectangle +(0.234675mm,0.234375mm);
\filldraw[black, thin] (6.5625mm,36.5625mm) rectangle +(0.234675mm,0.234375mm);
\filldraw[black, thin] (6.5625mm,45.9375mm) rectangle +(0.234675mm,0.234375mm);
\filldraw[black, thin] (6.5625mm,47.8125mm) rectangle +(0.234675mm,0.234375mm);
\filldraw[black, thin] (6.5625mm,49.6875mm) rectangle +(0.234675mm,0.234375mm);
\filldraw[black, thin] (6.5625mm,51.5625mm) rectangle +(0.234675mm,0.234375mm);

\filldraw[black, thin] (15.9375mm,0.9375mm) rectangle +(0.234675mm,0.234375mm);
\filldraw[black, thin] (15.9375mm,2.8125mm) rectangle +(0.234675mm,0.234375mm);
\filldraw[black, thin] (15.9375mm,4.6875mm) rectangle +(0.234675mm,0.234375mm);
\filldraw[black, thin] (15.9375mm,6.5625mm) rectangle +(0.234675mm,0.234375mm);
\filldraw[black, thin] (15.9375mm,15.9375mm) rectangle +(0.234675mm,0.234375mm);
\filldraw[black, thin] (15.9375mm,17.8125mm) rectangle +(0.234675mm,0.234375mm);
\filldraw[black, thin] (15.9375mm,19.6875mm) rectangle +(0.234675mm,0.234375mm);
\filldraw[black, thin] (15.9375mm,21.5625mm) rectangle +(0.234675mm,0.234375mm);
\filldraw[black, thin] (15.9375mm,30.9375mm) rectangle +(0.234675mm,0.234375mm);
\filldraw[black, thin] (15.9375mm,32.8125mm) rectangle +(0.234675mm,0.234375mm);
\filldraw[black, thin] (15.9375mm,34.6875mm) rectangle +(0.234675mm,0.234375mm);
\filldraw[black, thin] (15.9375mm,36.5625mm) rectangle +(0.234675mm,0.234375mm);
\filldraw[black, thin] (15.9375mm,45.9375mm) rectangle +(0.234675mm,0.234375mm);
\filldraw[black, thin] (15.9375mm,47.8125mm) rectangle +(0.234675mm,0.234375mm);
\filldraw[black, thin] (15.9375mm,49.6875mm) rectangle +(0.234675mm,0.234375mm);
\filldraw[black, thin] (15.9375mm,51.5625mm) rectangle +(0.234675mm,0.234375mm);

\filldraw[black, thin] (17.8125mm,0.9375mm) rectangle +(0.234675mm,0.234375mm);
\filldraw[black, thin] (17.8125mm,2.8125mm) rectangle +(0.234675mm,0.234375mm);
\filldraw[black, thin] (17.8125mm,4.6875mm) rectangle +(0.234675mm,0.234375mm);
\filldraw[black, thin] (17.8125mm,6.5625mm) rectangle +(0.234675mm,0.234375mm);
\filldraw[black, thin] (17.8125mm,15.9375mm) rectangle +(0.234675mm,0.234375mm);
\filldraw[black, thin] (17.8125mm,17.8125mm) rectangle +(0.234675mm,0.234375mm);
\filldraw[black, thin] (17.8125mm,19.6875mm) rectangle +(0.234675mm,0.234375mm);
\filldraw[black, thin] (17.8125mm,21.5625mm) rectangle +(0.234675mm,0.234375mm);
\filldraw[black, thin] (17.8125mm,30.9375mm) rectangle +(0.234675mm,0.234375mm);
\filldraw[black, thin] (17.8125mm,32.8125mm) rectangle +(0.234675mm,0.234375mm);
\filldraw[black, thin] (17.8125mm,34.6875mm) rectangle +(0.234675mm,0.234375mm);
\filldraw[black, thin] (17.8125mm,36.5625mm) rectangle +(0.234675mm,0.234375mm);
\filldraw[black, thin] (17.8125mm,45.9375mm) rectangle +(0.234675mm,0.234375mm);
\filldraw[black, thin] (17.8125mm,47.8125mm) rectangle +(0.234675mm,0.234375mm);
\filldraw[black, thin] (17.8125mm,49.6875mm) rectangle +(0.234675mm,0.234375mm);
\filldraw[black, thin] (17.8125mm,51.5625mm) rectangle +(0.234675mm,0.234375mm);

\filldraw[black, thin] (19.6875mm,0.9375mm) rectangle +(0.234675mm,0.234375mm);
\filldraw[black, thin] (19.6875mm,2.8125mm) rectangle +(0.234675mm,0.234375mm);
\filldraw[black, thin] (19.6875mm,4.6875mm) rectangle +(0.234675mm,0.234375mm);
\filldraw[black, thin] (19.6875mm,6.5625mm) rectangle +(0.234675mm,0.234375mm);
\filldraw[black, thin] (19.6875mm,15.9375mm) rectangle +(0.234675mm,0.234375mm);
\filldraw[black, thin] (19.6875mm,17.8125mm) rectangle +(0.234675mm,0.234375mm);
\filldraw[black, thin] (19.6875mm,19.6875mm) rectangle +(0.234675mm,0.234375mm);
\filldraw[black, thin] (19.6875mm,21.5625mm) rectangle +(0.234675mm,0.234375mm);
\filldraw[black, thin] (19.6875mm,30.9375mm) rectangle +(0.234675mm,0.234375mm);
\filldraw[black, thin] (19.6875mm,32.8125mm) rectangle +(0.234675mm,0.234375mm);
\filldraw[black, thin] (19.6875mm,34.6875mm) rectangle +(0.234675mm,0.234375mm);
\filldraw[black, thin] (19.6875mm,36.5625mm) rectangle +(0.234675mm,0.234375mm);
\filldraw[black, thin] (19.6875mm,45.9375mm) rectangle +(0.234675mm,0.234375mm);
\filldraw[black, thin] (19.6875mm,47.8125mm) rectangle +(0.234675mm,0.234375mm);
\filldraw[black, thin] (19.6875mm,49.6875mm) rectangle +(0.234675mm,0.234375mm);
\filldraw[black, thin] (19.6875mm,51.5625mm) rectangle +(0.234675mm,0.234375mm);

\filldraw[black, thin] (21.5625mm,0.9375mm) rectangle +(0.234675mm,0.234375mm);
\filldraw[black, thin] (21.5625mm,2.8125mm) rectangle +(0.234675mm,0.234375mm);
\filldraw[black, thin] (21.5625mm,4.6875mm) rectangle +(0.234675mm,0.234375mm);
\filldraw[black, thin] (21.5625mm,6.5625mm) rectangle +(0.234675mm,0.234375mm);
\filldraw[black, thin] (21.5625mm,15.9375mm) rectangle +(0.234675mm,0.234375mm);
\filldraw[black, thin] (21.5625mm,17.8125mm) rectangle +(0.234675mm,0.234375mm);
\filldraw[black, thin] (21.5625mm,19.6875mm) rectangle +(0.234675mm,0.234375mm);
\filldraw[black, thin] (21.5625mm,21.5625mm) rectangle +(0.234675mm,0.234375mm);
\filldraw[black, thin] (21.5625mm,30.9375mm) rectangle +(0.234675mm,0.234375mm);
\filldraw[black, thin] (21.5625mm,32.8125mm) rectangle +(0.234675mm,0.234375mm);
\filldraw[black, thin] (21.5625mm,34.6875mm) rectangle +(0.234675mm,0.234375mm);
\filldraw[black, thin] (21.5625mm,36.5625mm) rectangle +(0.234675mm,0.234375mm);
\filldraw[black, thin] (21.5625mm,45.9375mm) rectangle +(0.234675mm,0.234375mm);
\filldraw[black, thin] (21.5625mm,47.8125mm) rectangle +(0.234675mm,0.234375mm);
\filldraw[black, thin] (21.5625mm,49.6875mm) rectangle +(0.234675mm,0.234375mm);
\filldraw[black, thin] (21.5625mm,51.5625mm) rectangle +(0.234675mm,0.234375mm);

\filldraw[black, thin] (30.9375mm,0.9375mm) rectangle +(0.234675mm,0.234375mm);
\filldraw[black, thin] (30.9375mm,2.8125mm) rectangle +(0.234675mm,0.234375mm);
\filldraw[black, thin] (30.9375mm,4.6875mm) rectangle +(0.234675mm,0.234375mm);
\filldraw[black, thin] (30.9375mm,6.5625mm) rectangle +(0.234675mm,0.234375mm);
\filldraw[black, thin] (30.9375mm,15.9375mm) rectangle +(0.234675mm,0.234375mm);
\filldraw[black, thin] (30.9375mm,17.8125mm) rectangle +(0.234675mm,0.234375mm);
\filldraw[black, thin] (30.9375mm,19.6875mm) rectangle +(0.234675mm,0.234375mm);
\filldraw[black, thin] (30.9375mm,21.5625mm) rectangle +(0.234675mm,0.234375mm);
\filldraw[black, thin] (30.9375mm,30.9375mm) rectangle +(0.234675mm,0.234375mm);
\filldraw[black, thin] (30.9375mm,32.8125mm) rectangle +(0.234675mm,0.234375mm);
\filldraw[black, thin] (30.9375mm,34.6875mm) rectangle +(0.234675mm,0.234375mm);
\filldraw[black, thin] (30.9375mm,36.5625mm) rectangle +(0.234675mm,0.234375mm);
\filldraw[black, thin] (30.9375mm,45.9375mm) rectangle +(0.234675mm,0.234375mm);
\filldraw[black, thin] (30.9375mm,47.8125mm) rectangle +(0.234675mm,0.234375mm);
\filldraw[black, thin] (30.9375mm,49.6875mm) rectangle +(0.234675mm,0.234375mm);
\filldraw[black, thin] (30.9375mm,51.5625mm) rectangle +(0.234675mm,0.234375mm);

\filldraw[black, thin] (32.8125mm,0.9375mm) rectangle +(0.234675mm,0.234375mm);
\filldraw[black, thin] (32.8125mm,2.8125mm) rectangle +(0.234675mm,0.234375mm);
\filldraw[black, thin] (32.8125mm,4.6875mm) rectangle +(0.234675mm,0.234375mm);
\filldraw[black, thin] (32.8125mm,6.5625mm) rectangle +(0.234675mm,0.234375mm);
\filldraw[black, thin] (32.8125mm,15.9375mm) rectangle +(0.234675mm,0.234375mm);
\filldraw[black, thin] (32.8125mm,17.8125mm) rectangle +(0.234675mm,0.234375mm);
\filldraw[black, thin] (32.8125mm,19.6875mm) rectangle +(0.234675mm,0.234375mm);
\filldraw[black, thin] (32.8125mm,21.5625mm) rectangle +(0.234675mm,0.234375mm);
\filldraw[black, thin] (32.8125mm,30.9375mm) rectangle +(0.234675mm,0.234375mm);
\filldraw[black, thin] (32.8125mm,32.8125mm) rectangle +(0.234675mm,0.234375mm);
\filldraw[black, thin] (32.8125mm,34.6875mm) rectangle +(0.234675mm,0.234375mm);
\filldraw[black, thin] (32.8125mm,36.5625mm) rectangle +(0.234675mm,0.234375mm);
\filldraw[black, thin] (32.8125mm,45.9375mm) rectangle +(0.234675mm,0.234375mm);
\filldraw[black, thin] (32.8125mm,47.8125mm) rectangle +(0.234675mm,0.234375mm);
\filldraw[black, thin] (32.8125mm,49.6875mm) rectangle +(0.234675mm,0.234375mm);
\filldraw[black, thin] (32.8125mm,51.5625mm) rectangle +(0.234675mm,0.234375mm);

\filldraw[black, thin] (34.6875mm,0.9375mm) rectangle +(0.234675mm,0.234375mm);
\filldraw[black, thin] (34.6875mm,2.8125mm) rectangle +(0.234675mm,0.234375mm);
\filldraw[black, thin] (34.6875mm,4.6875mm) rectangle +(0.234675mm,0.234375mm);
\filldraw[black, thin] (34.6875mm,6.5625mm) rectangle +(0.234675mm,0.234375mm);
\filldraw[black, thin] (34.6875mm,15.9375mm) rectangle +(0.234675mm,0.234375mm);
\filldraw[black, thin] (34.6875mm,17.8125mm) rectangle +(0.234675mm,0.234375mm);
\filldraw[black, thin] (34.6875mm,19.6875mm) rectangle +(0.234675mm,0.234375mm);
\filldraw[black, thin] (34.6875mm,21.5625mm) rectangle +(0.234675mm,0.234375mm);
\filldraw[black, thin] (34.6875mm,30.9375mm) rectangle +(0.234675mm,0.234375mm);
\filldraw[black, thin] (34.6875mm,32.8125mm) rectangle +(0.234675mm,0.234375mm);
\filldraw[black, thin] (34.6875mm,34.6875mm) rectangle +(0.234675mm,0.234375mm);
\filldraw[black, thin] (34.6875mm,36.5625mm) rectangle +(0.234675mm,0.234375mm);
\filldraw[black, thin] (34.6875mm,45.9375mm) rectangle +(0.234675mm,0.234375mm);
\filldraw[black, thin] (34.6875mm,47.8125mm) rectangle +(0.234675mm,0.234375mm);
\filldraw[black, thin] (34.6875mm,49.6875mm) rectangle +(0.234675mm,0.234375mm);
\filldraw[black, thin] (34.6875mm,51.5625mm) rectangle +(0.234675mm,0.234375mm);

\filldraw[black, thin] (36.5625mm,0.9375mm) rectangle +(0.234675mm,0.234375mm);
\filldraw[black, thin] (36.5625mm,2.8125mm) rectangle +(0.234675mm,0.234375mm);
\filldraw[black, thin] (36.5625mm,4.6875mm) rectangle +(0.234675mm,0.234375mm);
\filldraw[black, thin] (36.5625mm,6.5625mm) rectangle +(0.234675mm,0.234375mm);
\filldraw[black, thin] (36.5625mm,15.9375mm) rectangle +(0.234675mm,0.234375mm);
\filldraw[black, thin] (36.5625mm,17.8125mm) rectangle +(0.234675mm,0.234375mm);
\filldraw[black, thin] (36.5625mm,19.6875mm) rectangle +(0.234675mm,0.234375mm);
\filldraw[black, thin] (36.5625mm,21.5625mm) rectangle +(0.234675mm,0.234375mm);
\filldraw[black, thin] (36.5625mm,30.9375mm) rectangle +(0.234675mm,0.234375mm);
\filldraw[black, thin] (36.5625mm,32.8125mm) rectangle +(0.234675mm,0.234375mm);
\filldraw[black, thin] (36.5625mm,34.6875mm) rectangle +(0.234675mm,0.234375mm);
\filldraw[black, thin] (36.5625mm,36.5625mm) rectangle +(0.234675mm,0.234375mm);
\filldraw[black, thin] (36.5625mm,45.9375mm) rectangle +(0.234675mm,0.234375mm);
\filldraw[black, thin] (36.5625mm,47.8125mm) rectangle +(0.234675mm,0.234375mm);
\filldraw[black, thin] (36.5625mm,49.6875mm) rectangle +(0.234675mm,0.234375mm);
\filldraw[black, thin] (36.5625mm,51.5625mm) rectangle +(0.234675mm,0.234375mm);

\filldraw[black, thin] (45.9375mm,0.9375mm) rectangle +(0.234675mm,0.234375mm);
\filldraw[black, thin] (45.9375mm,2.8125mm) rectangle +(0.234675mm,0.234375mm);
\filldraw[black, thin] (45.9375mm,4.6875mm) rectangle +(0.234675mm,0.234375mm);
\filldraw[black, thin] (45.9375mm,6.5625mm) rectangle +(0.234675mm,0.234375mm);
\filldraw[black, thin] (45.9375mm,15.9375mm) rectangle +(0.234675mm,0.234375mm);
\filldraw[black, thin] (45.9375mm,17.8125mm) rectangle +(0.234675mm,0.234375mm);
\filldraw[black, thin] (45.9375mm,19.6875mm) rectangle +(0.234675mm,0.234375mm);
\filldraw[black, thin] (45.9375mm,21.5625mm) rectangle +(0.234675mm,0.234375mm);
\filldraw[black, thin] (45.9375mm,30.9375mm) rectangle +(0.234675mm,0.234375mm);
\filldraw[black, thin] (45.9375mm,32.8125mm) rectangle +(0.234675mm,0.234375mm);
\filldraw[black, thin] (45.9375mm,34.6875mm) rectangle +(0.234675mm,0.234375mm);
\filldraw[black, thin] (45.9375mm,36.5625mm) rectangle +(0.234675mm,0.234375mm);
\filldraw[black, thin] (45.9375mm,45.9375mm) rectangle +(0.234675mm,0.234375mm);
\filldraw[black, thin] (45.9375mm,47.8125mm) rectangle +(0.234675mm,0.234375mm);
\filldraw[black, thin] (45.9375mm,49.6875mm) rectangle +(0.234675mm,0.234375mm);
\filldraw[black, thin] (45.9375mm,51.5625mm) rectangle +(0.234675mm,0.234375mm);

\filldraw[black, thin] (47.8125mm,0.9375mm) rectangle +(0.234675mm,0.234375mm);
\filldraw[black, thin] (47.8125mm,2.8125mm) rectangle +(0.234675mm,0.234375mm);
\filldraw[black, thin] (47.8125mm,4.6875mm) rectangle +(0.234675mm,0.234375mm);
\filldraw[black, thin] (47.8125mm,6.5625mm) rectangle +(0.234675mm,0.234375mm);
\filldraw[black, thin] (47.8125mm,15.9375mm) rectangle +(0.234675mm,0.234375mm);
\filldraw[black, thin] (47.8125mm,17.8125mm) rectangle +(0.234675mm,0.234375mm);
\filldraw[black, thin] (47.8125mm,19.6875mm) rectangle +(0.234675mm,0.234375mm);
\filldraw[black, thin] (47.8125mm,21.5625mm) rectangle +(0.234675mm,0.234375mm);
\filldraw[black, thin] (47.8125mm,30.9375mm) rectangle +(0.234675mm,0.234375mm);
\filldraw[black, thin] (47.8125mm,32.8125mm) rectangle +(0.234675mm,0.234375mm);
\filldraw[black, thin] (47.8125mm,34.6875mm) rectangle +(0.234675mm,0.234375mm);
\filldraw[black, thin] (47.8125mm,36.5625mm) rectangle +(0.234675mm,0.234375mm);
\filldraw[black, thin] (47.8125mm,45.9375mm) rectangle +(0.234675mm,0.234375mm);
\filldraw[black, thin] (47.8125mm,47.8125mm) rectangle +(0.234675mm,0.234375mm);
\filldraw[black, thin] (47.8125mm,49.6875mm) rectangle +(0.234675mm,0.234375mm);
\filldraw[black, thin] (47.8125mm,51.5625mm) rectangle +(0.234675mm,0.234375mm);

\filldraw[black, thin] (49.6875mm,0.9375mm) rectangle +(0.234675mm,0.234375mm);
\filldraw[black, thin] (49.6875mm,2.8125mm) rectangle +(0.234675mm,0.234375mm);
\filldraw[black, thin] (49.6875mm,4.6875mm) rectangle +(0.234675mm,0.234375mm);
\filldraw[black, thin] (49.6875mm,6.5625mm) rectangle +(0.234675mm,0.234375mm);
\filldraw[black, thin] (49.6875mm,15.9375mm) rectangle +(0.234675mm,0.234375mm);
\filldraw[black, thin] (49.6875mm,17.8125mm) rectangle +(0.234675mm,0.234375mm);
\filldraw[black, thin] (49.6875mm,19.6875mm) rectangle +(0.234675mm,0.234375mm);
\filldraw[black, thin] (49.6875mm,21.5625mm) rectangle +(0.234675mm,0.234375mm);
\filldraw[black, thin] (49.6875mm,30.9375mm) rectangle +(0.234675mm,0.234375mm);
\filldraw[black, thin] (49.6875mm,32.8125mm) rectangle +(0.234675mm,0.234375mm);
\filldraw[black, thin] (49.6875mm,34.6875mm) rectangle +(0.234675mm,0.234375mm);
\filldraw[black, thin] (49.6875mm,36.5625mm) rectangle +(0.234675mm,0.234375mm);
\filldraw[black, thin] (49.6875mm,45.9375mm) rectangle +(0.234675mm,0.234375mm);
\filldraw[black, thin] (49.6875mm,47.8125mm) rectangle +(0.234675mm,0.234375mm);
\filldraw[black, thin] (49.6875mm,49.6875mm) rectangle +(0.234675mm,0.234375mm);
\filldraw[black, thin] (49.6875mm,51.5625mm) rectangle +(0.234675mm,0.234375mm);

\filldraw[black, thin] (51.5625mm,0.9375mm) rectangle +(0.234675mm,0.234375mm);
\filldraw[black, thin] (51.5625mm,2.8125mm) rectangle +(0.234675mm,0.234375mm);
\filldraw[black, thin] (51.5625mm,4.6875mm) rectangle +(0.234675mm,0.234375mm);
\filldraw[black, thin] (51.5625mm,6.5625mm) rectangle +(0.234675mm,0.234375mm);
\filldraw[black, thin] (51.5625mm,15.9375mm) rectangle +(0.234675mm,0.234375mm);
\filldraw[black, thin] (51.5625mm,17.8125mm) rectangle +(0.234675mm,0.234375mm);
\filldraw[black, thin] (51.5625mm,19.6875mm) rectangle +(0.234675mm,0.234375mm);
\filldraw[black, thin] (51.5625mm,21.5625mm) rectangle +(0.234675mm,0.234375mm);
\filldraw[black, thin] (51.5625mm,30.9375mm) rectangle +(0.234675mm,0.234375mm);
\filldraw[black, thin] (51.5625mm,32.8125mm) rectangle +(0.234675mm,0.234375mm);
\filldraw[black, thin] (51.5625mm,34.6875mm) rectangle +(0.234675mm,0.234375mm);
\filldraw[black, thin] (51.5625mm,36.5625mm) rectangle +(0.234675mm,0.234375mm);
\filldraw[black, thin] (51.5625mm,45.9375mm) rectangle +(0.234675mm,0.234375mm);
\filldraw[black, thin] (51.5625mm,47.8125mm) rectangle +(0.234675mm,0.234375mm);
\filldraw[black, thin] (51.5625mm,49.6875mm) rectangle +(0.234675mm,0.234375mm);
\filldraw[black, thin] (51.5625mm,51.5625mm) rectangle +(0.234675mm,0.234375mm);
\end{tikzpicture}

\caption{A visualization of the support of $w_k$ when $n=2$, $k=3$.}

\label{fig5}
\end{figure}

%-------------------------------------------------
% NEW SECTION: PROOF OF THE THEOREM
%-------------------------------------------------
\section{Proof of the Theorem}

We first need the following two lemmas.

\begin{lemma} \label{lem2}
For $k\geqslant 3$ and $x\in \bigcup_{m=3}^{\infty}Q_1^m$, we have $T_\mathcal{S}(w_k)(x)\geqslant \frac{k}{2(2^n-1)} \thinspace w_k(x)$.
\end{lemma}

\begin{proof}
Assuming that $x\in Q_1^m$ for some integer $m\geqslant3$, we have
\begin{equation} \label{eq1}
\begin{split}
    T_\mathcal{S}(w_k)(x) & =\sum_{Q\in S}\frac{1}{|Q|}\int_Qw_k(y)\thinspace dy \hspace{2mm} \mathbbm{1}_Q(x) \\
    & =\sum_{Q\in S_0}\frac{1}{|Q|}\int_Qw_k(y)\thinspace dy \hspace{2mm} \mathbbm{1}_Q(x) \\
    & =\sum_{N=0}^{\infty}2^{Nn}\int_{[0,2^{-N})^n}w_k(y)\thinspace dy \hspace{2mm} \mathbbm{1}_{[0,2^{-N})^n}(x) \\
    & =\sum_{N=0}^{k(m-1)}2^{Nn}\int_{[0,2^{-N})^n}w_k(y)\thinspace dy \\
    & =1+\sum_{i=1}^{m-1}\sum_{j=1}^{k}2^{(j+k(i-1))n}\thinspace w_k\bigl([0,2^{-j-k(i-1)})^n\bigr)
\end{split}
\end{equation}
where the fourth equality is due to the fact that $Q_1^m$ is at least $2^{k-1}\times2^{-km}=2^{-[k(m-1)+1]}$ units away from the origin in all directions.

By the construction of $w_k$, we have
\begin{equation} \label{eq2}
    w_k\bigl([0,2^{-j-k(i-1)})^n\bigr) =\sum_{\substack{P\in\mathscr{P}^i \\ P\subset [0,2^{-j-k(i-1)})^n}}w_k(P) = 2^{(k-j)n}\thinspace w_k(Q_1^i).
\end{equation}

From \eqref{eq1} and \eqref{eq2}, we get
\begin{equation*}
\begin{split}
    T_\mathcal{S}(w_k)(x) & =1+\sum_{i=1}^{m-1}\sum_{j=1}^{k}2^{(j+k(i-1))n}\thinspace2^{(k-j)n}\thinspace2^{-kin}\left(\frac{2^{kn}}{2^{(k-1)n}+1}\right)^i \\
    & =1+\sum_{i=1}^{m-1}\sum_{j=1}^{k}\left(\frac{2^{kn}}{2^{(k-1)n}+1}\right)^i \\
    & =1+k \cdot \frac{2^{(k-1)n}+1}{(2^n-1)2^{(k-1)n}-1}\cdot \left[\left(\frac{2^{kn}}{2^{(k-1)n}+1}\right)^m-\frac{2^{kn}}{2^{(k-1)n}+1}\right] \\
    & =1+ \frac{k}{2} \cdot \frac{2^{(k-1)n}+1}{(2^n-1)2^{(k-1)n}-1}\cdot\left(\frac{2^{kn}}{2^{(k-1)n}+1}\right)^m \\
    & \hspace{6.5mm} + \frac{k}{2} \cdot \frac{2^{(k-1)n}+1}{(2^n-1)2^{(k-1)n}-1}\cdot\left[\left(\frac{2^{kn}}{2^{(k-1)n}+1}\right)^m-\frac{2^{kn+1}}{2^{(k-1)n}+1}\right].
\end{split}
\end{equation*}
Since $\left(\frac{2^{kn}}{2^{(k-1)n}+1}\right)^m \geqslant \frac{2^{kn+1}}{2^{(k-1)n}+1}$ whenever $k\geqslant3$, $m\geqslant3$ and $n\geqslant1$, we have
$$
T_\mathcal{S}(w_k)(x)\geqslant \frac{k}{2(2^n-1)} \thinspace w_k(x),
$$
and this finishes the proof of Lemma \ref{lem2}.
\end{proof}

\hfill

Consider the weight
$$w(x)=\sum_{k=3}^{\infty} A_k \thinspace w_k(x-\vec{2^k}) \hspace{1mm} \mathbbm{1}_{\Gamma(k)}(x-\vec{2^k})$$
where $\Gamma(k)=\bigcup_{m=3}^{\infty}Q_1^m(k)$, $\vec{2^k}=(2^k,...\thinspace,2^k)\in \mathbb{R}^n$, and the constant $A_k$ is chosen so that $A_kw_k(\Gamma(k))=1$, i.e., $A_k=(w_k(\Gamma(k)))^{-1}=2^{(k-1)n}\bigl(2^{(k-1)n}+1\bigr)^2$.   It is clear that $\supp(w)=\bigcup_{k=3}^{\infty}(\vec{2^k}+\Gamma(k))$.

\hfill

\begin{lemma}\label{lem3}
$M(w)(x)\sim w(x)$ for almost every $x\in \supp(w)$.
\end{lemma}
\begin{proof}
Since $x\in \supp(w)$, we have $x\in (\vec{2^k}+\Gamma(k))$ for some integer $k\geqslant 3$. For any cube $Q\ni x$, if $\ell(Q)\leqslant2^{k-2}$, then we have
\begin{equation} \label{eq3}
\begin{split}
    \frac{1}{|Q|}\int_Q w(y) \thinspace dy & = \frac{1}{|Q|}\int_Q A_{k}w_k(y-\vec{2^k})\mathbbm{1}_{\Gamma(k)}(y-\vec{2^k}) \thinspace dy \\
    & = A_{k} \thinspace \frac{1}{|Q-\vec{2^k}|}\int_{Q-\vec{2^k}} w_k(z)\mathbbm{1}_{\Gamma(k)}(z) \thinspace dz \\
    & \leqslant A_{k} \thinspace M(w_k\mathbbm{1}_{\Gamma(k)})(x-\vec{2^k}) \\
    & \leqslant A_{k} \thinspace M(w_k)(x-\vec{2^k}) \\
    & \lesssim A_{k} w_k(x-\vec{2^k}) = w(x)
\end{split}
\end{equation}
where the last inequality is due to Lemma \ref{lem1}.

If $\ell(Q)>2^{k-2}$, there exists an integer $k_1$ such that $2^{k_1-1}<\ell(Q)\leqslant2^{k_1}$. This implies that

\begin{equation} \label{eq4}
\begin{split}
    \frac{1}{|Q|}\int_Q w(y) \thinspace dy & = \frac{1}{|Q|}\int_Q \sum_{l=3}^{k_1+1} A_{l}\thinspace w_l(y-\vec{2^l})\mathbbm{1}_{\Gamma(l)}(y-\vec{2^l}) \thinspace dy \\
    & \leqslant \frac{1}{2^{(k_1-1)n}}\sum_{l=3}^{k_1+1} A_{l}\thinspace w_l(\Gamma(l)) \\
    & \leqslant \frac{k_1}{2^{(k_1-1)n}} < 1 < A_k \cdot w_k(x-\vec{2^k}) \\
    & \leqslant A_k \cdot M(w_k)(x-\vec{2^k}) \lesssim w(x).
\end{split}
\end{equation}

From \eqref{eq3} and \eqref{eq4}, we obtain the desired conlusion.
\end{proof}

\hfill

\begin{proof}[Proof of Theorem \ref{thmA}] 
Let $v=\left(\frac{Mw}{w}\right)^pw$, and consider the function
$$
f=\sum_{k=3}^{\infty}\frac{1}{k^\epsilon}\thinspace A_k\thinspace w_k(x-\vec{2^k})
$$
where $\frac{1}{p'}<\epsilon<1$. Since $T_\mathcal{S}$ is self-adjoint, to show the unboundedness of $T_\mathcal{S}$, it suffices to disprove the inequality
$$
\|T_\mathcal{S}f\|_{L^{p'}(v^{1-p'})} \lesssim \|f\|_{L^{p'}(w^{1-p'})}.
$$

In fact, we have
\begin{equation*}
\begin{split}
    \|f\|_{L^{p'}(w^{1-p'})}^{p'} & =\sum_{k=3}^{\infty}A_k\thinspace\frac{1}{k^{\epsilon p'}}\int_\mathbb{R}w_k(x-2^k) \thinspace\mathbbm{1}_{\Gamma(k)}(x-2^k) \thinspace dx \\
    & =\sum_{k=3}^{\infty}A_k\thinspace\frac{1}{k^{\epsilon p'}}w_k(\Gamma(k)) \\
    & =\sum_{k=3}^{\infty}A_k\frac{1}{k^{\epsilon p'}}\cdot\frac{1}{2^{k-1}(2^{k-1}+1)^2} = \sum_{k=3}^{\infty}\frac{1}{k^{\epsilon p'}} < \infty
\end{split}
\end{equation*}
while, by Lemma \ref{lem2} and Lemma \ref{lem3}, we have
\begin{equation*}
\begin{split}
    \|T_\mathcal{S}f\|_{L^{p'}(v^{1-p'})}^{p'} & =\int_\mathbb{R}T_\mathcal{S}f(x)^{p'}\frac{w(x)}{Mw(x)^{p'}} \thinspace dx \\
    & \geqslant C \int_\mathbb{R}T_\mathcal{S}f(x)^{p'}\thinspace w(x)^{1-p'} dx \\
    & = C \sum_{k=3}^{\infty}A_k^{1-p'}\thinspace \int_{\vec{2^k}+\Gamma(k)}T_\mathcal{S}f(x)^{p'}w_k(x-\vec{2^k})^{1-p'}\thinspace dx \\
    & = C \sum_{k=3}^{\infty}A_k^{1-p'}\thinspace \int_{\Gamma(k)}T_\mathcal{S}f(x+\vec{2^k})^{p'}w_k(x)^{1-p'}\thinspace dx \\
    & \geqslant C \sum_{k=3}^{\infty}A_k k^{(1-\epsilon)p'} w_k(\Gamma(k)) =\sum_{k=3}^{\infty}k^{(1-\epsilon)p'} = \infty.
\end{split}
\end{equation*}

\hfill

In the fifth estimate, we used the fact that whenever $x\in\Gamma(k)$, one has
\begin{equation*}
\begin{split}
    T_\mathcal{S}f(x+\vec{2^k}) & =\sum_{Q\in S_k}\frac{1}{|Q|}\int_Qf(y)\thinspace dy \hspace{1mm} \mathbbm{1}_Q(x+\vec{2^k}) \\
    & =\sum_{P\in S_0}\frac{1}{|P|}\int_{\vec{2^k}+P}f(y)\thinspace dy \hspace{1mm} \mathbbm{1}_P(x) \\
    & =\sum_{P\in S_0}\frac{1}{|P|}\int_{\vec{2^k}+P}\frac{1}{k^\epsilon}\thinspace A_k\thinspace w_k(y-\vec{2^k})\thinspace dy \hspace{1mm} \mathbbm{1}_P(x) \\
    & =\frac{1}{k^\epsilon}\thinspace A_k\thinspace\sum_{P\in S_0}\frac{1}{|P|}\int_{P}w_k(z)\thinspace dz \hspace{1mm} \mathbbm{1}_P(x) \\
    & =\frac{1}{k^\epsilon}\thinspace A_k\thinspace T_\mathcal{S}(w_k)(x)
\end{split}
\end{equation*}
which makes Lemma \ref{lem2} still applicable in the situation.

\hfill

Using Lemma \ref{lem3}, the boundedness of the maximal operator follows exactly the same as in \cite{CS}. For the reader's convenience, we will summarize the arguments as following.
\begin{multline*}
    \|Mf\|^p_{L^p(w)}  =\int_{\mathbb{R}^n} Mf(x)^pw(x)\thinspace dx \lesssim \int_{\mathbb{R}^n} |f(x)|^p\thinspace Mw(x)\thinspace dx \\
     \lesssim \int_{\mathbb{R}^n} |f(x)|^p \left(\frac{Mw(x)}{w(x)}\right)^pw(x)\thinspace dx = \|f\|^p_{L^p(v)}
\end{multline*}
where the first inequality is due to Fefferman and Stein \cite{FS} (remember that $\supp f\subset \supp w$).

Similarly, we have
\begin{multline*}
    \|Mf\|^{p'}_{L^{p'}(v^{1-p'})}  =\int_{\mathbb{R}^n} Mf(x)^{p'}v(x)^{1-p'}\thinspace dx  \approx \int_{\mathbb{R}^n} Mf(x)^{p'}w(x)^{1-p'} dx \\
     \lesssim \int_{\mathbb{R}^n} |f(x)|^{p'} w(x)^{1-p'} dx = \|f\|^{p'}_{L^{p'}(w^{1-p'})}
\end{multline*}

where the last inequality was proved in \cite{CS} (cf. Proof of Theorem 3).
\end{proof}

\hfill

%-------------------------------------------------
% REFERENCES
%-------------------------------------------------
\bibliographystyle{amsplain}

\hfill

\end{document}